\documentclass[a4paper,twoside]{article}

\usepackage{un}

\author{C\'edric Bonnaf\'e\footnote{The author is partly supported by 
the ANR (Project No JC07-192339).}~ and Gregor Kemper}

\title{Some complete intersection symplectic quotients in positive
  characteristic: invariants of a vector and a covector}

\date{February 21, 2011}

\begin{document}

\maketitle

\begin{abstract}
  Given a linear action of a group $G$ on a $K$-vector space $V$, we
  consider the invariant ring $K[V \oplus V^*]^G$, where $V^*$ is the
  dual space. We are particularly interested in the case where $V =
  \gfq^n$ and $G$ is the group $U_n$ of all upper unipotent matrices
  or the group $B_n$ of all upper triangular matrices in
  $\GL_n(\gfq)$.
  
  In fact, we determine $\gfq[V \oplus V^*]^G$ for $G = U_n$ and $G =
  B_n$. The result is a complete intersection for all values of~$n$
  and~$q$. We present explicit lists of generating invariants and
  their relations. This makes an addition to the rather short list of
  ``doubly parametrized'' series of group actions whose invariant
  rings are known to have a uniform description.
\end{abstract}

\section*{Introduction}

Many interesting subgroups of $\GL_n(\gfq)$ come in doubly
parametrized series, where one parameter is linked to~$n$ and the
other to~$q$. Important examples are the finite classical groups, the
groups $B_n$ and $U_n$ of upper triangular matrices and unipotent
upper triangular matrices in $\GL_n(\gfq)$, and the cyclic $p$-groups
acting indecomposably. In the context of invariant theory, not only
the natural actions but also others, including decomposable ones, are
interesting.  For the following series of groups with their natural
actions, the invariant rings have been determined: the general and
special linear groups (this goes back to L. Dickson, see for instance 
\mycite[Chapter~8.1]{lsm} or \mycite{wil}),
the groups $B_n$ and $U_n$ (see \mycite[Section~4.5, Example~2]{NS} or
\mycite[Proposition~5.5.6]{lsm}), the finite symplectic groups (this
goes back to D. Carlisle and P.  Kropholler, see
\mycite[Chapter~8.3]{bens}), and the finite unitary groups
(\mycite{Chi:Yao:2006}). For $\GL_n(\gfq)$, $\SL_n(\gfq)$, $U_n$, and
$B_n$, the invariant rings are isomorphic to polynomials rings, and
their determination is fairly easy. For the finite symplectic and
unitary groups, the invariant rings are complete intersections, and
the same is expected for the finite orthogonal groups
(see~[\citenumber{Chi:Yao:2006}]). To the best of our knowledge, no
results have appeared so far about the invariant rings of a
doubly-parametrized series of groups with a non-trivial decomposable action.

In this paper we study the invariant rings of the type $K[V \oplus
V^*]^G$, where $G$ is a finite group acting on a finite-dimensional
$K$-vector space $V$ and $V^*$ is the dual space. In the language of
classical invariant theory, the elements of $K[V \oplus V^*]^G$ are
called {\em invariants of a vector and a covector}. In the case that
$K$ has characteristic zero and $G$ is generated by reflections, $K[V
\oplus V^*]^G$ has been studied intensively in the last fifteen years,
in relation with the representation theory of Cherednik algebras and
the geometry of Hilbert schemes and Calogero-Moser spaces: see the
pioneering work of \citename{Haiman:94} on the symmetric group case
[\citenumber{Haiman:94}], [\citenumber{Haiman:01}],
[\citenumber{Haiman:02}] and, for instance, \mycite{Etingof:Ginzburg},
\mycite{Ginzburg:Kaledin}, \mycite{Gordon:03}, and
\mycite{Bellamy:09}. The ring $K[V \oplus V^*]^G$ is also important
for the computation of invariants in Weyl algebras (see
\mycite{Kemper:Quiring}). Here we consider the case that $K = \gfq$ is
a finite field and $G$ is one of the groups $B_n$ or $U_n$, and
calculate the invariant ring $\gfq[V \oplus V^*]^G$.
The result is surprisingly simple. In fact, writing $\gfq[V \oplus
V^*] = \gfq[x_1 \upto x_n,y_1 \upto y_n]$ (where $x_1 \upto x_n$ is
the standard basis of $V$ and $y_1 \upto y_n$ is the dual basis) and
setting
\[
\renewcommand{\arraystretch}{1.9}
\begin{array}{llllll}
  f_i & := \displaystyle \prod_{h \in U_n \cdot x_i} h, \qquad & f_i^*
  & := \displaystyle \prod_{h \in U_n \cdot y_{n+1-i}} h \qquad & (1
  \le i \le n), \\
  \fti_i & := f_i^{q-1}, \qquad & \fti_i^* & := f_i^{* q-1} \qquad &
  (1 \le i \le n), \\
  u_j & : = \displaystyle \sum_{k=1}^n x_k^{q^j} y_k, \qquad
  \text{and} \qquad & u_{-j} & := \displaystyle \sum_{k=1}^n x_k
  y_k^{q^j} \qquad & (j \ge 0),
\end{array}
\]
we prove:

{\em
  \begin{enumerate}
  \item \label{0a} If $n \ge 2$, then $\gfq[V \oplus
    V^*]^{U_n}=\gfq[f_1 \upto f_n,f_1^* \upto f_n^*,u_{2-n} \upto
    u_{n-2}]$ is generated by $4n-3$ invariants subject to $2n-3$
    relations.
  \item \label{0b} If $n \ge 1$, then $\gfq[V \oplus
    V^*]^{B_n}=\gfq[\fti_1 \upto \fti_n,\fti_1^* \upto
    \fti_n^*,u_{1-n} \upto u_{n-1}]$ is generated by $4n-1$ invariants
    subject to $2n-1$ relations.
  \end{enumerate}
  The relations are given explicitly in \tref{main}. In particular,
  both $\gfq$-algebras $\gfq[V \oplus V^*]^{U_n}$ and $\gfq[V \oplus
  V^*]^{B_n}$ are complete intersections.
} \\

The special case $n = 2$ and~$q$ a prime of~\eqref{0a} is included in
\mycite{Mara.Neusel}. Observe that the number of generators and the
number of relations are independent of~$q$.

Notice that by a result of \mycite{KW} and \mycite{Gordeev:87}, the
invariant ring $K[V \oplus V^*]^G$ can only be a complete intersection
if $G$ is generated by pseudo-reflections. However, even when $G$ is
generated by pseudo-reflections, it seems to be rare that $K[V \oplus
V^*]^G$ is a complete intersection. A counterexample, possibly the
smallest, is given by the symmetric group $S_3$ acting irreducibly on
$V = \CC^2$. We checked that by using the computer algebra system
MAGMA (see~[\citenumber{magma}]). See also \mycite{Alev:Foissy}.

Another indication that the invariant rings $\gfq[V \oplus V^*]^{U_n}$
and $\gfq[V \oplus V^*]^{B_n}$ are ``lucky'' cases comes from
comparing them to $\gfq[V \oplus V]^{U_n}$ and $\gfq[V \oplus
V]^{B_n}$. Using MAGMA, we find that for $n = 3$ and $q = 2$ or~$3$,
$\gfq[V \oplus V]^{U_3}$ requires a minimum of 12 or 16 generators,
respectively, and fails to be Cohen--Macaulay for $q = 3$. \\

The paper is organized as follows: in the first section we start by
determining the invariant field $K(V \oplus V^*)^G$ for all finite
groups $G \le \GL(V)$ for which $K[V]^G$ and $K[V^*]^G$ is known.
Then we prove a lemma (see~\ref{1lA}) which gives a sufficient
condition for a $K$-algebra to admit a particular presentation by
generators and relations. This lemma will be used for all results in
this paper. The main part of the paper is the second section, where we
produce relations between our claimed generators, and show that they
satisfy the hypotheses of \lref{1lA}. This leads to the main result,
\tref{main}. In the final section we study the invariant ring 
$\gfq[V \oplus V^*]^G$ for $G=\SL_n(\gfq)$ or $\GL_n(\gfq)$. 
We make a conjecture about $\gfq[V \oplus V^*]^{\GL_n(\gfq)}$
(see~\ref{3cGL}). \\

We should mention the role of experimental work in the genesis of this
paper. The starting point was the explicit computation of $\FF_q[V
\oplus V^*]^{U_n}$ for $n = 3$ and $q = 2,3$ (and its approximate
computation for $q = 4,5$) by using MAGMA. This prompted us to guess
the generators of $\FF_q[V \oplus V^*]^{U_n}$ for $n = 3$. By
obtaining the relations appearing in \exref{1ex23}($U_3$) and using
\lref{1lA}, we were able to prove the case $n = 3$ of
\tref{main}\eqref{maina}. Turning to the case $n = 4$, we used MAGMA
again to produce some relations between our conjectured generators for
several~$q$. From these, we guessed (and verified) the relations for
general~$q$ appearing in \exref{1ex23}($U_4$). We observed that these
relations again satisfy the hypotheses of \lref{1lA}. We then pushed
this up to $n = 5$ and~$6$. Only then were we able to conjecture the
general relations given in \tref{main}\eqref{maina} and to observe
that they can be interpreted as special cases of the determinant
identity from \lref{2lDet}. This led to the (computer-free) proof of
part~\eqref{maina} of \tref{main}, and part~\eqref{mainb} was then
deduced quite easily. So it is justified to say that this paper owes
its existence to MAGMA.

\section{Preliminaries} \label{1sPreliminaries}

Let $K$ be a field, $n$ a positive integer, and $V=K^n$. The general
linear group $\GL_n(K)$ acts naturally on $V$. It also acts on the
dual space $V^*$ by $\sigma \cdot \lambda := \lambda \circ
\sigma^{-1}$ for $\sigma \in \GL_n(K)$ and $\lambda \in V^*$. This
induces an action on the polynomial ring $K[V \oplus V^*]$, which by
convention we take to be the symmetric algebra of $V \oplus V^*$.
(Since $V \oplus V^*$ is self-dual, the more standard convention of
taking the symmetric algebra of the dual yields the same result.) We
can write
\[
K[V \oplus V^*] = K[x_1 \upto x_n,y_1 \upto y_n],
\]
where $x_1 \upto x_n$ is the standard basis of $V = K^n$ and 
$y_1\upto y_n$ is the dual basis. 

The natural pairing
\[
V \otimes V^* \to K, \ v \otimes \lambda \mapsto \lambda(v)
\]
is clearly invariant under the action of $\GL_n(K)$. Since $V
\otimes V^*$ is embedded into $K[V \oplus V^*]$, this gives rise
to an invariant~$u_0$. Explicitly, we obtain
\[
u_0 = \sum_{j=1}^n x_j y_j \in K[V \oplus V^*]^{\GL_n(K)}.
\]

We start by looking at the invariant field $K(V \oplus V^*)^G$.
Recall that for some important finite subgroups $G \subseteq
\GL_n(K)$, generators of the invariant ring $K[V]^G$ are known. If $K$
is finite, these subgroups include $U_n$, $B_n$, $\SL_n(K)$, and
$\GL_n(K)$ (see \mycite[Proposition~5.5.6 and Theorems~8.1.5
and~8.1.8]{lsm}).

\begin{prop} \label{1pField}
  Let $G \subseteq \GL_n(K)$ be a finite subgroup. Then $K(V \oplus
  V^*)^G$ is generated, as a field extension of $K$, by $K[V]^G$,
  $K[V^*]^G$, and $u_0$.
\end{prop}

\begin{proof}
  Let $f_1 \upto f_l$ (respectively $g_1 \upto g_m$) be generators of
  the $K$-algebra $K[V]^G$ (respectively $K[V^*]^G$). The group $G
  \times G$ acts in the obvious way on $V \oplus V^*$, and it follows
  that
  \[
  K[V \oplus V^*]^{G \times G} = K\left[f_1 \upto f_l,g_1 \upto
    g_m\right].
  \]
  So $K(V \oplus V^*)$ is Galois as a field extension of $K\left(f_1
    \upto f_l, g_1,\dots, g_m\right)$ with group $G \times G$. It
  follows that it is also Galois as a field extension of $L :=
  K\left(f_1 \upto f_l, g_1,\dots, g_m, u_0\right)$. Clearly $L
  \subseteq K(V \oplus V^*)^G$, so if we can show that the Galois
  group $\gal{K(V \oplus V^*)/L}$ is contained in $G$ embedded diagonally, then Galois
  theory yields $K(V \oplus V^*)^G = L$.
  
  So take an arbitrary element from this Galois group $\gal{K(V \oplus
    V^*)/L}$, which we can write as $(\sigma,\tau) \in G \times G$. We
  need to show that $\sigma = \tau$. We have
  \[
  (\sigma \tau^{-1},\id)(u_0) =(\sigma
  \tau^{-1},\id)\left((\tau,\tau)(u_0)\right) = (\sigma,\tau)(u_0) =
  u_0.
  \]
  Since the $y_i$ are algebraically independent over $K[x_1 \upto
  x_n]$, this shows that $(\sigma \tau^{-1})(x_j) = x_j$ for all~$j$,
  so $\sigma \tau^{-1} = \id$. This concludes the proof.
\end{proof}

We have an involution
\[
\mapl{*}{K[V \oplus V^*]}{K[V \oplus V^*]}{x_i}{y_{n+1-i}},\ y_i
\mapsto x_{n+1-i}.
\]
For $\sigma \in \GL_n(K)$ we set
\[
\sigma^* :=
\begin{pmatrix}
0 & \cdots & 1 \\
\vdots & \text{\reflectbox{$\ddots$}} & \vdots \\
1 & \cdots & 0
\end{pmatrix} \cdot (\sigma^{-1})^T \cdot \begin{pmatrix}
0 & \cdots & 1 \\
\vdots & \text{\reflectbox{$\ddots$}} & \vdots \\
1 & \cdots & 0
\end{pmatrix}.
\]
It is easy to verify that for $\sigma \in \GL_n(K)$ and $f \in K[V
\oplus V^*]$, the rule
\[
(\sigma \cdot f)^* = \sigma^* \cdot f^*
\]
holds. So if $G \subseteq \GL_n(K)$ is stable under the automorphism
$*$ of $\GL_n(K)$, then $*$ induces an automorphism of the invariant
ring $K[V \oplus V^*]^G$, and this automorphism restricts to an
isomorphism between $K[V]^G$ and $K[V^*]^G$.

\begin{ex}
  The groups $U_n$, $B_n$, $\SL_n(K)$ and $\GL_n(K)$ are $*$-stable.
\end{ex}

We obtain the following corollary from \pref{1pField}.

\begin{cor} \label{1cField}
  Let $G \subseteq \GL_n(K)$ be a $*$-stable finite subgroup. Assume
  that $K[V]^G$ is generated by the invariants $f_1 \upto f_m$ (as a
  $K$-algebra). Then $K(V \oplus V^*)^G$ is generated (as a field
  extension of $K$) by $f_i$, $f_i^*$ ($i = 1 \upto m$), and~$u_0$.
\end{cor}

For the proof of our main results we will use the following lemma. It
gives a sufficient condition for a $K$-algebra to admit a particular
presentation by generators and relations.

\begin{lemma} \label{1lA}
  Let $A$ be a graded algebra over $K$. Suppose that $A$ is an
  integral domain. Let $f_1 \upto f_n,g_1 \upto g_m,h_1 \upto h_l \in
  A$ be homogeneous elements of positive degree such that
  \begin{enumerate}
  \item \label{1lAa} $f_1 \upto f_n,g_1 \upto g_m$ form a homogeneous
    system of parameters of $A$ (e.i., they are algebraically
    independent and $A$ is an integral extension of the subalgebra
    formed by them) and
  \item \label{1lAb} for the field of fractions we have
    \[
    \Quot(A) = K(f_1 \upto f_n,g_1 \upto g_m,h_1 \upto h_l).
    \]
  \end{enumerate}
  Moreover, let $R_1 \upto R_l$ be homogeneous elements of the kernel
  of the homomorphism
  \[
  \mapl{\phi}{P := K[X_1 \upto X_n,Y_1 \upto Y_m,Z_1 \upto
    Z_l]}{A}{X_i}{f_i}, \ Y_i \mapsto g_i, \ Z_i \mapsto h_i,
  \]
  were $P$ is a polynomial ring graded in such a way that~$\phi$ is
  degree-preserving. Suppose that
  \begin{enumerate}
    \setcounter{enumi}{2}
  \item \label{1lAc} $X_1 \upto X_n,Y_1 \upto Y_m,R_1 \upto R_l$ form
    a homogeneous system of parameters of $P$ and
  \item \label{1lAd}%
    for
    \[
    x := \overline{X}_1 \cdots \overline{X}_n \in B := P/(R_1 \upto
    R_l)
    \]
    (where $\overline{X}_i$ denotes the class in $B$ of $X_i$), the
    localization $B_x := B[x^{-1}]$ is generated by $x^{-1}$,
    $\overline{X}_1 \upto \overline{X}_n$, and~$m$ further elements.
    Moreover, for $y := \overline{Y}_1 \cdots \overline{Y}_m$, $B_y$
    is generated by $y^{-1}$, $\overline{Y}_1 \upto \overline{Y}_m$,
    and~$n$ further elements. (Loosely speaking, this means that after
    localizing by~$x$ or~$y$, the relations allow us to eliminate~$l$
    of the generators.)
  \end{enumerate}
  Then
  \[
  A = K[f_1 \upto f_n,g_1 \upto g_m,h_1 \upto h_l].
  \]
  Moreover, $A$ is a complete intersection, and the kernel of~$\phi$
  is generated by $R_1 \upto R_l$.
\end{lemma}

\begin{proof}
  The first goal is to show that $B$ is an integral domain. We
  conclude from~\eqref{1lAc} that
  \begin{equation} \label{1eqDim0}
    \dim\left(\strut P/(R_1 \upto R_l,X_1 \upto X_n,Y_1 \upto
      Y_m)\right) = 0.
  \end{equation}
  Therefore $B$ is a complete intersection of dimension $n + m$.
  In particular, $B$ is Cohen--Macaulay (see
  \mycite[Proposition~18.13]{eis}). It follows from~\eqref{1eqDim0}
  that for every $i \in \{1 \upto n\}$, $\overline{X}_i$ lies in no
  minimal prime ideal of $B$. By the unmixedness theorem (see
  [\citenumber{eis}, Corollary~18.14]), all associated prime ideals of
  $(0)$ are minimal, so it follows that $\overline{X}_i$ is a
  non-zero-divisor. Since this holds for all~$i$, also~$x$ is a
  non-zero-divisor. Therefore $B$ embeds into $B_x$. In particular,
  $B_x$ has transcendence degree at least $n + m$. So it follows
  from~\eqref{1lAd} that $B_x$ is a localized polynomial ring and in
  particular an integral domain. This implies that $B$ is also an
  integral domain. Similarly, $B_y$ is a localized polynomial ring.
  This will be used in a moment.
  
  Now we show that $B$ is normal. Let ${\mathfrak p} \in \Spec(B)$ be
  a prime ideal of height one. It follows from~\eqref{1eqDim0} that
  for all $i \in \{1 \upto n\}$ and $j \in \{1 \upto m\}$ the ideal
  $(\overline{X}_i,\overline{Y}_j) \subseteq B$ has height~$2$.
  Therefore $\mathfrak p$ cannot contain both $x$ and $y$, so
  $B_{\mathfrak p}$ is a localization of $B_x$ or of $B_y$ and
  therefore normal. This shows that $B$ satisfies Serre's condition R1
  (see~[\citenumber{eis}, Theorem~11.5]).  Moreover, applying the
  unmixedness theorem again, we see that $B$ also satisfies the
  condition S2 (see [\citenumber{eis}, Theorem~11.5]). By Serre's
  criterion ([\citenumber{eis}, Theorem~11.5]), $B$ is normal.
  
  Consider the epimorphism
  \[
  \map{\psi}{B}{A' := K[f_1 \upto f_n,g_1 \upto g_m,h_1 \upto h_l]
    \subseteq A}
  \]
  induced from~$\phi$. It follows from~\eqref{1lAa} that $A'$ has
  dimension $n + m$, the same as $B$. Since $B$ is an integral domain,
  it follows that $\ker(\psi) = \{0\}$, so~$\psi$ is an isomorphism.
  In particular, $A'$ is normal. Applying~\eqref{1lAa} again, we see
  that $A$ is integral over $A'$. But by~\eqref{1lAb}, $A \subseteq
  \Quot(A')$, so the normality of $A'$ implies $A = A'$. We have
  already seen that $A' \cong B$ is a complete intersection. The
  injectivity of~$\psi$ means that the kernel of~$\phi$ is generated
  by the $R_i$. So the proof is complete.
\end{proof}

Readers may find it helpful to take a look at \exref{1ex23} already
now. There, \lref{1lA} is applied several times, so the example serves
to illustrate the less intuitive hypotheses~\eqref{1lAc}
and~\eqref{1lAd} of the lemma.

\section{The invariant ring of $U_n$ and $B_n$}


From now on, we assume that $K = \FF_q$ is a finite field with~$q$
elements.

\paragraph{Some invariants.} The homomorphisms
\begin{equation} \label{1eqFxy}
  \begin{aligned}
    \mapl{F}{& \gfq[V \oplus V^*]}{\gfq[V \oplus V^*]}{x_i}{x_i^q},
    \ y_i \mapsto y_i, \qquad \text{and} \\
    \mapl{F^*}{& \gfq[V \oplus V^*]}{\gfq[V \oplus V^*]}{x_i}{x_i},
    \ y_i \mapsto y_i^q,
  \end{aligned}
\end{equation}
commute with the action of $\GL_n(\gfq)$.
Therefore we get further invariants in $\gfq[V \oplus
V^*]^{\GL_n(\gfq)}$ by setting, for $i \ge 0$,
\[
u_i := F^i(u_0) = \sum_{j=1}^n x_j^{q^i} y_j \qquad \text{and} \qquad
u_{-i} := (F^*)^i(u_0) = \sum_{j=1}^n x_j y_j^{q^i} .
\]
Notice that $u_{-i} = u_i^*$ for all $i \in \ZZ$.

Now we turn our attention to the case where $G \in \{U_n,B_n\}$.
Apart from the invariants $u_i$ defined above, we get obvious
invariants by taking the orbit-products (for $1 \le i \le n$)
\[
f_i := \prod_{h \in U_n \cdot x_i} h = \prod_{\alpha_1 \upto
  \alpha_{i-1} \in \gfq} \left(\strut x_i + \smash{\sum_{j=1}^{i-1}}
  \alpha_j x_j\right).
\]
Then
\[
f_i^* = \prod_{h \in U_n \cdot y_{n+1-i}} h = \prod_{\alpha_1 \upto
  \alpha_{i-1} \in \gfq} \left(\strut y_{n+1-i} +
  \smash{\sum_{j=1}^{i-1}} \alpha_j y_{n+1-j}\right).
\]
The $f_i$ and $f^*_i$ are homogeneous of degrees $\deg(f_i) = \deg(f^*_i)
= q^{i-1}$. Similarly, we set (for $1 \le i \le n$)
\[
\fti_i := f_i^{q-1}=-\prod_{h \in B_n \cdot x_i} h,
\]
so that
\[
\fti_i^* = (f_i^*)^{q-1}=-\prod_{h \in B_n \cdot y_{n+1-i}} h .
\]
The minus sign comes from the fact that $\prod_{\xi \in {\mathbb{F}}^\times_{\! q}} \xi = -1$. 
It is well known (see \mycite[Section~4.5, Example~2]{NS} or
\mycite[Proposition~5.5.6]{lsm}) that
\begin{equation}\label{classical invariants}
  \gfq[V]^{U_n} = \gfq[f_1,\dots,f_n] \qquad \text{and} \qquad
  \gfq[V]^{B_n}=\gfq[\fti_1,\dots,\fti_n].
\end{equation}
So if we want to use \lref{1lA} for showing that the $f_i$, $f_i^*$
(respectively, $\fti_i$ and $\fti_i^*$) together with some $u_i$
generate the invariant ring, the hypotheses~\eqref{1lAa}
and~\eqref{1lAb} are already satisfied. So everything hinges on our
ability to find some suitable relations between the generators.

\paragraph{Some relations.} The following identity provides the source
of our relations.

\begin{lemma} \label{2lDet}
  Let $R$ be a commutative ring with identity element, $n$ a positive
  integer, and $a_{i,j}, b_{i,j} \in R$ ($i,j = 1 \upto n$).
  Then for $1 \le k \le n$ we have
  \begin{equation} \label{2eqDet}
    \begin{aligned}
      \sum_{i=1}^k \sum_{j=1}^{n+1-k} \sum_{l=1}^n & (-1)^{i+j+n+1}
      a_{i,l} b_{j,n+1-l} \cdot \det
      \begin{pmatrix}
        a_{1,1} & \cdots & a_{1,k-1} \\
        \vdots & & \vdots \\
        a_{i-1,1} & \cdots & a_{i-1,k-1} \\
        a_{i+1,1} & \cdots & a_{i+1,k-1} \\
        \vdots & & \vdots \\
        a_{k,1} & \cdots & a_{k,k-1}
      \end{pmatrix} \\
      & \cdot \det
      \begin{pmatrix}
        b_{1,1} & \cdots & b_{1,n-k} \\
        \vdots & & \vdots \\
        b_{j-1,1} & \cdots & b_{j-1,n-k} \\
        b_{j+1,1} & \cdots & b_{j+1,n-k} \\
        \vdots & & \vdots \\
        b_{n+1-k,1} & \cdots & b_{n+1-k,n-k}
      \end{pmatrix} \\
      & = \det
      \begin{pmatrix}
        a_{1,1} & \cdots & a_{1,k} \\
        \vdots & & \vdots \\
        a_{k,1} & \cdots & a_{k,k}
      \end{pmatrix} \cdot \det
      \begin{pmatrix}
        b_{1,1} & \cdots & b_{1,n+1-k} \\
        \vdots & & \vdots \\
        b_{n+1-k,1} & \cdots & b_{n+1-k,n+1-k}
      \end{pmatrix}.
    \end{aligned}
  \end{equation}
  In the case $k = 1$, the first determinant in the left-hand side
  of~\eqref{2eqDet} is to be understood as~$1$, and in the case $k =
  n$, the second determinant is to be understood as~$1$.
\end{lemma}

\begin{proof}
  First, observe that, for $1 \le l \le n$,
  \begin{equation} \label{2eqX}
    \sum_{i=1}^k (-1)^i a_{i,l} \cdot \det
    \begin{pmatrix}
      a_{1,1} & \cdots & a_{1,k-1} \\
      \vdots & & \vdots \\
      a_{i-1,1} & \cdots & a_{i-1,k-1} \\
      a_{i+1,1} & \cdots & a_{i+1,k-1} \\
      \vdots & & \vdots \\
      a_{k,1} & \cdots & a_{k,k-1}
    \end{pmatrix}
    = (-1)^k \det
    \begin{pmatrix}
      a_{1,1} & \cdots & a_{1,k-1} & a_{1,l} \\
      \vdots & & \vdots & \vdots\\
      a_{k,1} & \cdots & a_{k,k-1} & a_{k,l} \\
    \end{pmatrix}
  \end{equation}
  and
  \begin{equation} \label{2eqY}
    \begin{aligned}
      \sum_{j=1}^{n+1-k} (-1)^j & b_{j,n+1-l} \cdot \det
      \begin{pmatrix}
        b_{1,1} & \cdots & b_{1,n-k} \\
        \vdots & & \vdots \\
        b_{j-1,1} & \cdots & b_{j-1,n-k} \\
        b_{j+1,1} & \cdots & b_{j+1,n-k} \\
        \vdots & & \vdots \\
        b_{n+1-k,1} & \cdots & b_{n+1-k,n-k}
      \end{pmatrix} \\
      & = (-1)^{n+1-k} \det
      \begin{pmatrix}
        b_{1,1} & \cdots & b_{1,n-k} & b_{1,n+1-l} \\
        \vdots & & \vdots & \vdots \\
        b_{n+1-k,1} & \cdots & b_{n+1-k,n-k} & b_{n+1-k,n+1-l}
      \end{pmatrix}.
    \end{aligned}
  \end{equation}
  Moreover, the right-hand side of~\eqref{2eqX}
  (respectively~\eqref{2eqY}) is zero if $l \le k-1$ (respectively $l
  \ge k+1$). So by multiplying~\eqref{2eqX} and~\eqref{2eqY} and
  summing over $l = 1 \upto n$, we obtain~\eqref{2eqDet}. Notice that
  the special cases $k = 1$ and $k = n$ pose no problems in the proof.
\end{proof}

We apply \lref{2lDet} to $R = \gfq[V \oplus V^*]$, $a_{i,j} =
x_j^{q^{i-1}}$, and $b_{i,j} = y_{n+1-j}^{q^{i-1}} =
\left(x_j^{q^{i-1}}\right)^*$. We wish to express the relations
obtained in this way in terms of the invariants $u_i$, $f_i$,
$\fti_i$, $f_i^*$, and $\fti_i^*$. First, note that the sums $\sum_{l
  = 1}^n a_{i,l} b_{j,n+1-l}$ in~\eqref{2eqDet} specialize to
\[
\sum_{l = 1}^n x_l^{q^{i-1}} y_l^{q^{j-1}} =
u_{i-j}^{q^{\min\{i-1,j-1\}}}.
\]
Therefore, setting
\[
d_{k,i} := \det
\begin{pmatrix}
  x_1 & x_2 & \cdots & x_k \\
  x_1^q & x_2^q & \cdots & x_k^q \\
  \vdots & \vdots & & \vdots \\
  x_1^{q^{i-1}} & x_2^{q^{i-1}} & \cdots & x_k^{q^{i-1}} \\
  x_1^{q^{i+1}} & x_2^{q^{i+1}} & \cdots & x_k^{q^{i+1}} \\
  \vdots & \vdots & & \vdots \\
  x_1^{q^k} & x_2^{q^k} & \cdots & x_k^{q^k}
\end{pmatrix}
\]
and shifting the summations indices~$i$ and~$j$ in~\eqref{2eqDet} down
by~1, we obtain
\begin{equation} \label{2eqD}
  \sum_{i=0}^{k-1} \sum_{j=0}^{n-k} (-1)^{i+j+n+1}
  u_{i-j}^{q^{\min\{i,j\}}} \cdot d_{k-1,i} \cdot d_{n-k,j}^* =
  d_{k,k} \cdot d_{n+1-k,n+1-k}^*
\end{equation}
for $1 \le k \le n$. The following lemma expresses the determinants
$d_{k,i}$ in terms of our invariants.

\begin{lemma} \label{2lC}
  For $1 \le k \le n$ and $0 \le i \le k$ we have
  \[
  d_{k,i} = \prod_{j=1}^k f_j \cdot \left(\sum_{1 \le j_1 < j_2 <
      \cdots < j_{k-i} \le k} \,\,\prod_{l=1}^{k-i}
    \fti_{j_l}^{q^{i+l-j_l}}\right).
  \]
  For $i = k$, the sum on the right-hand side should be interpreted
  as~$1$.
\end{lemma}

\begin{proof}
  Most of the ideas in the proof are taken from \mycite{wil}. We first
  treat the case $i = k$ using induction on~$k$. We have
  \[
  d_{1,1} = x_1 = f_1.
  \]
  Now we go from~$k$ to $k+1$. Substituting $x_{k+1} = \sum_{j=1}^k
  \alpha_j x_j$ with $\alpha_j \in \FF_q$ into $d_{k+1,k+1}$
  yields~$0$.  Since the $x_{k+1}$-degree of $d_{k+1,k+1}$ is $q^k$,
  we conclude that as polynomials in $x_{k+1}$, both $d_{k+1,k+1}$ and
  $f_{k+1}$ have the same roots. So they are equal up to a factor in
  $\FF_q(x_1 \upto x_k)$. By comparing leading coefficients, we see
  that
  \begin{equation} \label{2eqbkk}
    d_{k+1,k+1} = d_{k,k} \cdot f_{k+1}.
  \end{equation}
  (This equation even holds for $k = n$ if we define $f_{n+1} :=
  \prod_{\alpha_1 \upto \alpha_n \in \FF_q} \left(x_{n+1} +
    \sum_{j=1}^n \alpha_j x_j\right)$ with an additional indeterminate
  $x_{n+1}$.) From~\eqref{2eqbkk}, we obtain the desired result for
  $d_{k+1,k+1}$ by induction.
  
  Expanding the determinant $d_{k+1,k+1}$ along the last column gives
  \[
  d_{k+1,k+1} = \sum_{i=0}^k (-1)^{k+i} d_{k,i} x_{k+1}^{q^i}.
  \]
  So by~\eqref{2eqbkk} we can write
  \[
  f_{k+1} = \sum_{i=0}^k (-1)^{k+i} c_{k,i} x_{k+1}^{q^i}
  \]
  with $c_{k,i} := d_{k,i}/d_{k,k} \in \FF_q[x_1 \upto x_k]$. So we
  need to show that
  \begin{equation} \label{2eqC}
    c_{k,i} = \sum_{1 \le j_1 < j_2 < \cdots < j_{k-i} \le k}
    \,\,\prod_{l=1}^{k-i} \fti_{j_l}^{q^{i+l-j_l}}.
  \end{equation}
  Again we use induction on~$k$, this time starting with $k = 0$. We
  have $f_1 = x_1$, so $c_{0,0} = 1$ as claimed. For $0 < k \le n$ we
  have
  \begin{align*}
    f_{k+1} & = \prod_{\alpha_1 \upto \alpha_k \in \FF_q} \left(\strut
      x_{k+1} + \smash{\sum_{j=1}^k \alpha_j x_j}\right) =
    \prod_{\alpha_k \in \FF_q} f_k(x_1 \upto x_{k-1},x_{k+1}+\alpha_k x_k) \\
    & = \prod_{\alpha_k \in \FF_q} \left(\strut f_k(x_1 \upto
      x_{k-1},x_{k+1}) + \alpha_k f_k(x_1 \upto x_{k-1},x_k)\right) \\
    & = f_k(x_1 \upto x_{k-1},x_{k+1})^q - f_k(x_1 \upto
    x_{k-1},x_{k+1}) \cdot \fti_k \\
    & = \sum_{i=0}^{k-1} (-1)^{k+i+1} \left(c_{k-1,i}^q \cdot
      x_{k+1}^{q^{i+1}} - \fti_k c_{k-1,i} \cdot
      x_{k+1}^{q^i}\right).
  \end{align*}
  This yields the recursive formula
  \[
  c_{k,i} = c_{k-1,i-1}^q + \fti_k c_{k-1,i},
  \]
  where we set $c_{k-1,-1} = c_{k-1,k} := 0$. For $i = k$ we have
  $c_{k,i} = 1$, satisfying~\eqref{2eqC} by convention. For $0 < i <
  k$ we use induction and obtain
  \begin{align*}
    c_{k,i} & = \! \sum_{1 \le j_1 < j_2 < \cdots < j_{k-i} \le k-1}
    \,\,\prod_{l=1}^{k-i} \left(\fti_{j_l}^{q^{i+l-j_l-1}}\right)^q
    + \!\! \sum_{1 \le j_1 < j_2 < \cdots < j_{k-i-1} \le k-1}
    \prod_{l=1}^{k-i-1} \fti_{j_l}^{q^{i+l-j_l}} \fti_k \\
    & = \sum_{1 \le j_1 < j_2 < \cdots < j_{k-i} \le k}
    \,\,\prod_{l=1}^{k-i} \fti_{j_l}^{q^{i+l-j_l}}
  \end{align*}
  as desired. (No problem arises in the special case $i = k-1$.) For
  $i = 0$ we obtain
  \[
  c_{k,i} = \fti_k c_{k-1,0} = \fti_k \cdot \prod_{l=1}^{k-1} \fti_l =
  \prod_{l=1}^k \fti_l.
  \]
  This completes the proof.
\end{proof}

It follows from \lref{2lC} that both sides of~\eqref{2eqD} are
divisible by $\prod_{j=1}^{k-1} f_j \cdot \prod_{j=1}^{n-k}
f_j^*$. Therefore, setting
\begin{equation} \label{2eqDickson}
  c_{s,t} := \sum_{1 \le j_1 < j_2 < \cdots < j_{s-t} \le s}
  \,\,\prod_{l=1}^{s-t} \fti_{j_l}^{q^{t+l-j_l}} = \sum_{1 \le j_1 <
    j_2 < \cdots < j_{s-t} \le s} \,\,\prod_{l=1}^{s-t}
  f_{j_l}^{q^{t+l-j_l} (q-1)}
\end{equation}
for $0 \le t < s \le n$ and $c_{s,s} := 1$ for $0 \le s \le n$, we
obtain the relation
\begin{equation*} \tag{$R_k$} \label{2eqRk}
  \sum_{i=0}^{k-1} \sum_{j=0}^{n-k} (-1)^{i+j+n+1} c_{k-1,i} \cdot
  c_{n-k,j}^* \cdot u_{i-j}^{q^{\min\{i,j\}}} - f_k \cdot f_{n+1-k}^*
  = 0
\end{equation*}
for $1 \le k \le n$. We deduce some further relations
from~\eqref{2eqRk} by applying the homomorphisms $F$ and $F^*$ (see
\eqref{1eqFxy}). This yields
\begin{align*} \tag{$R_k^+$} \label{2eqRkp}
  \sum_{i=0}^{k-1} \sum_{j=0}^{n-k} (-1)^{i+j+n+1} c_{k-1,i}^q \cdot
  c_{n-k,j}^* \cdot u_{i-j+1}^{q^{\min\{i+1,j\}}} - f_k^q \cdot
  f_{n+1-k}^* & = 0 \\
  \intertext{and} \tag{$R_k^-$} \label{2eqRkm} \sum_{i=0}^{k-1}
  \sum_{j=0}^{n-k} (-1)^{i+j+n+1} c_{k-1,i} \cdot c_{n-k,j}^{* q}
  \cdot u_{i-j-1}^{q^{\min\{i,j+1\}}} - f_k \cdot f_{n+1-k}^{* q} & =
  0.
\end{align*}
The relations produced so far involve the $U_n$-invariants $f_i$,
$f_i^*$, and $u_i$. In order to obtain relations between the
$B_n$-invariants $\fti_i$, $\fti_i^*$, and $u_i$, we raise $f_k \cdot
f_{n+1-k}^*$ and the remaining sum in~\eqref{2eqRk} to the $(q-1)$st
power. This yields
\begin{equation*} \tag{$\tilde{R}_k$} \label{3eqRk}
  \left(\sum_{i=0}^{k-1} \sum_{j=0}^{n-k} (-1)^{i+j} c_{k-1,i} \cdot
    c_{n-k,j}^* \cdot u_{i-j}^{q^{\min\{i,j\}}}\right)^{q-1} -
  \fti_k \cdot \fti_{n+1-k}^* = 0.
\end{equation*}
Furthermore, by subtracting the $\tilde{f}_k$-fold of~\eqref{2eqRk}
from~\eqref{2eqRkp}, we obtain
\begin{equation*} \tag{$\tilde{R}_k^+$} \label{3eqRkp}
  \sum_{i=0}^{k-1} \sum_{j=0}^{n-k} (-1)^{i+j} \left(c_{k-1,i}^q \cdot
    c_{n-k,j}^* \cdot u_{i-j+1}^{q^{\min\{i+1,j\}}} - \fti_k \cdot
    c_{k-1,i} \cdot c_{n-k,j}^* \cdot u_{i-j}^{q^{\min\{i,j\}}}\right) =
  0.
\end{equation*}

\begin{rem}
  \begin{enumerate}
  \item It may be of interest that $c_{s,t}$ is the $t$th Dickson
    invariant in $x_1 \upto x_s$ (see \mycite[Section~8.1]{lsm} or
    \mycite{wil}).  This follows from the proof of \lref{2lC}.
  \item It is easy to see that the relations~\eqref{2eqRk},
    \eqref{2eqRkp}, \eqref{2eqRkm}, \eqref{3eqRk}, and~\eqref{3eqRkp}
    are homogeneous.  (Their degrees are listed on
    page~\pageref{2RelDegs}.)  \remend
  \end{enumerate}
  \renewcommand{\remend}{}
\end{rem}

\paragraph{Main result.} 
We are now ready to prove the main result of this paper.

\begin{theorem} \label{main}
  With the above notation, we have:
  \begin{enumerate}
  \item \label{maina} If $n \ge 2$, then
    \[
    \gfq[V \oplus V^*]^{U_n} = \gfq[f_1,\dots,f_n,f_1^* \upto
    f_n^*,u_{2-n} \upto u_{n-2}]
    \]
    is generated by $4n - 3$ invariants. If $n \ge 3$, the ideal of
    relations has the following $2n - 3$ generators:
    \begin{equation} \label{2eqRels}
      R_1^+,R_2,R_3^-,R_3,R_4^-,R_4,R_5^- \upto
      R_{n-2},R_{n-1}^-,R_{n-1},R_n^-.
    \end{equation}
    If $n=2$, then the ideal of relations is generated by
    \begin{equation} \label{1eqR2}
      u_0^q - (f_1 f_1^*)^{q-1} u_0 - f_1^q f_2^* - f_1^{*q} f_2 = 0.
    \end{equation}
  \item \label{mainb} If $n \ge 1$, then
    \[
    \gfq[V \oplus V^*]^{B_n} = \gfq[\fti_1 \upto \fti_n,\fti_1^* \upto
    \fti_n^*,u_{1-n} \upto u_{n-1}]
    \]
    is generated by $4n-1$ elements, and the ideal of relations has
    the following $2n - 1$ generators:
    \begin{equation} \label{3eqRels}
      \tilde{R}_1,\tilde{R}_1^+,\tilde{R}_2,\tilde{R}_2^+ \upto
      \tilde{R}_{n-1},\tilde{R}_{n-1}^+,\tilde{R}_n.
    \end{equation}
  \end{enumerate}
  In particular, both $\gfq$-algebras $\gfq[V \oplus V^*]^{U_n}$ and
  $\gfq[V \oplus V^*]^{B_n}$ are complete intersections. The
  generating invariants given in~\eqref{maina} and~\eqref{mainb} are
  minimal, except in the case $q = 2$ of~\eqref{mainb} (in which $B_n
  = U_n$).
\end{theorem}

Before proving \tref{main}, we shall provide examples in the case
where $n \in \{1,2,3,4\}$.

\begin{ex} \label{1ex23}
  \begin{enumerate}
  \item[($U_1$)] If $n=1$, then $U_n = U_1 = \{1\}$ and $\gfq[V \oplus
    V^*]^{U_1} = \gfq[V \oplus V^*] = \gfq[x_1,y_1]$. This case is not
    covered by the uniform description of \tref{main}\eqref{maina}.
  \item[($U_2$)] If $n=2$, we have
    \[
    f_1 = x_1, \quad f_1^* = y_2, \quad f_2 = x_2^q - x_2 x_1^{q-1},
    \quad f_2^* = y_1^q - y_1 y_2^{q-1}, \quad \text{and} \quad u_0 =
    x_1y_1+x_2y_2.
    \]
    The relation~\eqref{1eqR2} can be verified by direct computation:
    \begin{multline*}
      \left(x_1 y_1 + x_2 y_2\right)^q - (x_1 y_2)^{q-1} \left(x_1 y_1
        + x_2 y_2\right) - x_1^q \left(y_1^q - y_1 y_2^{q-1}\right) -
      y_2^q \left(x_2^q - x_2 x_1^{q-1}\right) = \\
      x_1^q y_1^q + x_2^q y_2^q - x_1^q y_1 y_2^{q-1} - x_1^{q-1} x_2
      y_2^q - x_1^q y_1^q + x_1^q y_1 y_2^{q-1} - x_2^q y_2^q +
      x_1^{q-1} x_2 y_2^q = 0.
    \end{multline*}
    We have already seen that $f_1$, $f_1^*$, $f_2$, $f_2^*$, and
    $u_0$ satisfy the hypotheses~\eqref{1lAa} and~\eqref{1lAb} from
    \lref{1lA} (see after~\eqref{classical invariants}).  The
    relation~\eqref{1eqR2} also satisfies~\eqref{1lAc} from
    \lref{1lA}. Indeed, if we treat $u_0$ and the $f_i$ and $f^*_i$ as
    indeterminates for a moment, it is clear that the relation
    together with $f_1,f_2,f^*_1$, and $f^*_2$ forms a homogeneous
    system of parameters. Moreover, if we localize by $f_1$, the
    relation can be used to eliminate $f^*_2$ as a generator; and
    localizing by $f^*_1$ eliminates the generator $f_2$.
    So~\eqref{1lAd} is also satisfied, and applying \lref{1lA} proves
    \tref{main}\eqref{maina} for $n = 2$.
    
    Why are the relations for $n = 2$ not given by~\eqref{2eqRels}?
    Notice that the relations $R_1$, $R_2^-$ read
    \[
    - f_1^{* q-1} u_0 + u_{-1} - f_1f_2^* = 0 \quad \text{and} \quad -
    f_1^{q-1} u_{-1} + u_0^q - f_2 f_1^{*q} = 0,
    \]
    so they involve $u_{-1}$, which is not included in the list of
    generators. But \eqref{1eqR2} can be obtained by adding the
    $f_1^{q-1}$-fold of $R_1$ to $R_2^-$.
  \item[($U_3$)] For $n = 3$, the relations are
    \begin{align*}
      \tag{$R_1^+$} \label{2eqR31} u_{-1}^q - (f_1^{*q
        (q-1)}+f_2^{*q-1}) u_0^q + (f_1^* f_2^*)^{q-1} u_1 - f_1^q
      f_3^* & = 0, \\ 
      \tag{$R_2$} \label{2eqR32} u_0^q - f_1^{*q-1} u_1 - f_1^{q-1}
      u_{-1} + (f_1 f_1^*)^{q-1} u_0 - f_2 f_2^* & = 0, \\
      \tag{$R_3^-$} \label{2eqR33} u_1^q - (f_1^{q (q-1)}+f_2^{q-1})
      u_0^q + (f_1 f_2)^{q-1} u_{-1} - f_3 f_1^{*q} & = 0.
    \end{align*}
    It is clear that the relations satisfy~\eqref{1lAc} from
    \lref{1lA}. Moreover, if we localize the algebra $B$ defined by
    the relations by $f_1^* f_2^*$, we obtain an algebra that is
    generated by $f_1^*,f_2^*,f_3^*,f_1,u_{-1},u_0$, and $(f_1^*
    f_2^*)^{-1}$. (By abuse of notation, we write $f_i^*$ for the
    element corresponding to $f_i^*$ in $B$ and so on.) In fact, we
    can use~\eqref{2eqR31} to eliminate the generator $u_1$ of
    $B_{f_1^* f_2^*}$, then~\eqref{2eqR32} to eliminate $f_2$ and,
    finally, \eqref{2eqR33} to eliminate $f_3$. We can also localize
    by $f_1 f_2$.  Then we use~\eqref{2eqR33} to eliminate $u_{-1}$,
    then~\eqref{2eqR32} to eliminate $f_2^*$, and, finally,
    \eqref{2eqR31} to eliminate $f^*_3$. So we are left with the
    generators $f_1,f_2,f_3,f_1^*,u_0,u_1$, and $(f_1 f_2)^{-1}$.
    
    There is a total of nine relations of the type~($R_k^{(\pm)}$),
    but as it happens, just the above three serve for the proof of
    \tref{main}\eqref{maina} in the case $n = 3$. Besides, some of the
    nine relations involve invariants other than
    $f_1,f_2,f_3,f_1^*,f^*_2,f^*_3,u_{-1},u_0,u_1$. For example,
    \eqref{2eqR1} reads
    \begin{equation*} \tag{$R_1$} \label{2eqR1}
      u_{-2} - (f_1^{*q (q-1)}+f_2^{*q-1}) u_{-1} + (f_1^* f_2^*)^{q-1}
      u_0 - f_1 f_3^* = 0,
    \end{equation*}
    so it serves to express $u_{-2}$ in terms of the above nine
    invariants.
  \item[($U_4$)] For $n = 4$, the relations from \tref{main} (a) read
    \begin{align*}
      \tag{$R_1^+$} \label{2eqR1p} & \begin{aligned} u_{-2}^q & -
        \left(f_1^{*q^2 (q-1)} + f_2^{*q (q-1)} +
          f_3^{*q-1}\right) u_{-1}^q \\
        & + \left((f^*_1 f^*_2)^{q (q-1)} + (f_1^{*q} f_3^*)^{q-1} +
          (f_2^* f_3^*)^{q-1}\right) u_0^q - (f_1^* f_2^* f^*_3)^{q-1}
        u_1 - f_1^q f^*_4 = 0,
    \end{aligned} \\ \\
    \tag{$R_2$} \label{2eqR2} & \begin{aligned} u_{-1}^q & - f_1^{q-1}
      u_{-2} - \left(f_1^{*q (q-1)} + f_2^{*q-1}\right) u_0^q +
      f_1^{q-1} \left(f_1^{*q (q-1)} +
        f_2^{*q-1}\right) u_{-1} \\
      & + (f_1^* f^*_2)^{q-1} u_1 - (f_1 f^*_1 f^*_2)^{q-1} u_0 - f_2
      f^*_3 = 0,
    \end{aligned} \\ \\
    \tag{$R_3^-$} \label{2eqR3m} & \begin{aligned} u_0^{q^2} & -
      f_1^{*q (q-1)} u_1^q - \left(f_1^{q (q-1)} + f_2^{q-1}\right)
      u_{-1}^q + \left(f_1^{q (q-1)} +
        f_2^{q-1}\right) f_1^{*q (q-1)} u_0^q \\
      & + (f_1 f_2)^{q-1} u_{-2} - (f_1 f_2 f_1^{*q})^{q-1} u_{-1} -
      f_3 f_2^{*q} = 0,
    \end{aligned} \\ \\
    \tag{$R_3$} \label{2eqR3} & \begin{aligned} u_1^q & - f_1^{*q-1}
      u_2 - \left(f_1^{q (q-1)} + f_2^{q-1}\right)
      u_0^q + \left(f_1^{q (q-1)} + f_2^{q-1}\right) f_1^{*q-1} u_1 \\
      & + (f_1 f_2)^{q-1} u_{-1} - (f_1 f_2 f_1^*)^{q-1} u_0 - f_3
      f_2^* = 0,
    \end{aligned} \\ \\
    \tag{$R_4^-$} \label{2eqR4m} & \begin{aligned} u_2^q & -
      \left(f_1^{q^2 (q-1)} + f_2^{q (q-1)} +
        f_3^{q-1}\right)  u_1^q \\
      & + \left((f_1 f_2)^{q (q-1)} + (f_1^q f_3)^{q-1} + (f_2
        f_3)^{q-1}\right) u_0^q - (f_1 f_2 f_3)^{q-1} u_{-1} - f_4
      f_1^{*q} = 0.
    \end{aligned}
  \end{align*}
  With these relations, we can make an argument analogous to the above
  for $U_3$, showing that \lref{1lA} is applicable. We will do this in
  general in the forthcoming proof of \tref{main}\eqref{maina}.
  
  Notice that applying the involution~$*$ transforms~\eqref{2eqR1p}
  into~\eqref{2eqR4m} and~\eqref{2eqR2} into~\eqref{2eqR3};
  but~\eqref{2eqR3m} is not invariant under~$*$.  This ``violation of
  symmetry'' can be fixed by adding the $(f_2^{q-1})$-fold
  of~\eqref{2eqR2} to~\eqref{2eqR3m}. The result is
  \begin{equation*} \tag{$R_3'$} \label{2eqR3s}
    \begin{aligned}
      u_0^{q^2} & - f_1^{*q (q-1)} u_1^q - f_1^{q (q-1)} u_{-1}^q +
      \left((f_1 f^*_1)^{q (q-1)} - (f_2 f^*_2)^{q-1}\right) u_0^q \\
      & + (f_2 f^*_1 f^*_2)^{q-1} u_1 + (f_1 f_2 f^*_2)^{q-1} u_{-1} - (f_1
      f_2 f^*_1 f^*_2)^{q-1} u_0 - f_3 f_2^{*q} - f_2^q f^*_3=0,
    \end{aligned}
  \end{equation*}
  which is $*$-invariant and can substitute the
  relation~\eqref{2eqR3m}. This also demonstrates that there is some
  arbitrariness in our choice of generating relations.
  \item[($B_1$)] If $n=1$, then $\fti_1 = x_1^{q-1}$, $\fti_1^* =
    y_1^{q-1}$, and $u_0 = x_1 y_1$. \tref{main}\eqref{mainb} asserts
    that $\gfq[V \oplus V^*]^{B_1}$ is generated by $\fti_1$,
    $\fti_1^*$, and $u_0$, subject to the relation
    \begin{equation*} \tag{$\tilde{R}_1$}
      u_0^{q-1} - \fti_1\fti_1^* = 0.
    \end{equation*}
    This can easily be verified by hand.
  \item[($B_2$)] If $n=2$, one gets the following three relations
    between the $B_2$-invariants:
    \begin{align*}
      \tag{$\tilde{R}_1$} \left(u_{-1} - \fti_1^* u_0\right)^{q-1} -
      \fti_1 \fti^*_2 & = 0, \\
      \tag{$\tilde{R}_1^+$} u_0^q - \fti_1 u_{-1} - \fti_1^*
      u_1 + \fti_1 \fti_1^* u_0 & = 0, \\
      \tag{$\tilde{R}_2$} \left(u_1 - \fti_1 u_0\right)^{q-1} - \fti_2
      \fti_1^* & = 0.
    \end{align*}
    This looks nicely symmetric, in the sense that the set of
    relations is stable under the involution~$*$. But it is clear that
    the symmetry will be lost when~$n$ becomes bigger. In fact, our
    choice of generating relations of the $B_n$-invariants is
    arbitrary, just as in the case of $U_n$-invariants. \exend
  \end{enumerate}
  \renewcommand{\exend}{}
\end{ex}

\begin{proof}[Proof of \tref{main}]
  The proofs of~\eqref{maina} and~\eqref{mainb} are very similar and
  both rely on the use of Lemma \ref{1lA}.
  \medskip
  
  $\bullet$ Let us first prove~\eqref{maina}. Since
  \exref{1ex23}($U_2$) deals with the case where $n = 2$, we may
  assume that $n > 2$.  We want to apply \lref{1lA} to $A = \gfq[V
  \oplus V^*]^{U_n}$, $m = n$, $l = 2n-3$, $g_i = f_i^*$, $(h_1 \upto
  h_l)=(u_{2-n} \upto u_{-1},u_0,u_1 \upto u_{n-2})$, and $R_1 \upto
  R_l$ being replaced by $R_1^+$, $R_2$, $R_3^-$, $R_3$, $R_4^-$,
  $R_4$, $R_5^-$ \upto $R_{n-2}$, $R_{n-1}^-$, $R_{n-1}$, $R_n^-$.

  From \eqref{classical invariants} we deduce that $f_1 \upto f_n$,
  $f_1^* \upto f_n^*$ satisfy the hypothesis~\eqref{1lAa} from
  \lref{1lA}. From \cref{1cField} and again \eqref{classical
    invariants}, it follows that
  \[
  \gfq(V \oplus V^*)^{U_n} = \gfq(f_1 \upto f_n,f^*_1 \upto
  f^*_n,u_0),
  \]
  so the hypothesis~\eqref{1lAb} of \lref{1lA} is also satisfied.
  
  In order to establish the hypotheses~\eqref{1lAc} and~\eqref{1lAd}
  of \lref{1lA}, we analyze the relations~($R_k^{(\pm)}$).
  
  We will say that one of the relations is a relation {\em for} a
  $u_i$ if the relation equates a power of $u_i$ to a polynomial in
  our claimed generators, and each monomial of this polynomial
  involves at least one of the~$f_i$ or~$f_i^*$. We will say that one
  of the relations {\em $f$-eliminates} a (claimed) generator~$g$ if
  this relation, viewed as a polynomial in~$g$, has degree~$1$ and
  leading coefficient a product of powers of the~$f_i$. In the same
  way, we speak of relations that {\em $f^*$-eliminate} generators.
  Notice that $c_{s,0}$, as defined in \eqref{2eqDickson}, is a
  product of powers of the~$f_i$, and $c_{s,0}^*$ is a product of
  powers of the~$f_i^*$. Using this terminology, our analysis of the
  relations can be summarized in the following table:
  \begin{center}
    \begin{tabular}{|c|c|c|c|c|c|}
      \hline \strut relation & involves & relation for & $f$-eliminates &
      $f^*$-eliminates & range \\ \hline
      \ref{2eqRkp} & \parbox{2.7cm}{\strut $f_1 \upto f_k$, $f^*_1 \upto
        f^*_{n+1-k}$, $u_{k-n+1} \upto u_k$} & $u_{2 k - n}$ &
      $f^*_{n+1-k}$ & $u_k$ & $k = 1$ \\ \hline
      \ref{2eqRk} & \parbox{2.7cm}{\strut $f_1 \upto f_k$, $f^*_1 \upto
        f^*_{n+1-k}$, $u_{k-n} \upto u_{k-1}$} & $u_{2 k-n-1}$ &
      $f^*_{n+1-k},u_{k-n}$ & $f_k,u_{k-1}$ & $2 \le k \le n-1 $ \\
      \hline
      \ref{2eqRkm} & \parbox{2.7cm}{\strut $f_1 \upto f_k$, $f^*_1 \upto
        f^*_{n+1-k}$, $u_{k-n-1} \upto u_{k-2}$} & $u_{2 k-n-2}$ &
      $u_{k-n-1}$ & $f_k$ & $3 \le k \le n$ \\ \hline
    \end{tabular}
  \end{center}
  The last column of the table indicates the range of~$k$ specified
  in~\eqref{2eqRels}. We make several observations.
  
  First, since $n > 2$, the relations in~\eqref{2eqRels} involve the
  invariants $f_1 \upto f_n$, $f^*_1 \upto f^*_n$, and $u_{2-n} \upto
  u_{n-2}$, which are exactly the generators claimed in
  \tref{main}\eqref{maina}.
  
  Second, in~\eqref{2eqRels} we have one relation for every $u_i$
  (with $2-n \le i \le n-2$). If we regard the $f_i$, $f^*_i$, and $u_i$
  as indeterminates for a moment, it follows that the affine variety
  in $\overline{\gfq}^{4 n - 3}$ given by the equations $f_i = 0$,
  $f^*_i = 0$ and the relations in~\eqref{2eqRels} consists of only one
  point, the origin. It follows that the hypothesis~\eqref{1lAc} of
  \lref{1lA} is satisfied.
  
  It remains to show that~\eqref{1lAd} is also satisfied. By another
  abuse of notation, we will now regard the $f_i$, $f^*_i$, and $u_i$ as
  elements of the algebra $B$ defined by the relations
  in~\eqref{2eqRels}.
  We can use the relations
  \[
  \setlength{\arraycolsep}{0.25mm}
  \begin{array}{lllllllllllr}
    R_1^+, & R_2, & R_3^-, & R_3, & R_4^-, & R_4, & R_5^- \upto & 
    R_{n-2}, & R_{n-1}^-, & R_{n-1}, & R_n^- & \quad \text{(in this
      order) to show that} \\
    u_1, & f_2, & f_3, & u_2, & f_4, & u_3, & f_5 \upto & u_{n-3}, &
    f_{n-1}, & u_{n-2}, & f_n & \quad \text{(also in this order)}
  \end{array}
  \]
  lie in $\gfq\left[(f^*_1 \cdots f^*_n)^{-1},f^*_1 \upto f^*_n,f_1,u_{2-n}
    \upto u_0\right]$. So this algebra is equal to $B\left[(f^*_1 \cdots
    f^*_n)^{-1}\right]$. We can also use
  \[
  \setlength{\arraycolsep}{0.25mm}
  \begin{array}{lllllllllllr}
    R_n^-, & R_{n-1}, & R_{n-1}^-, & R_{n-2} \upto & R_5^-, & R_4, &
    R_4^-, & R_3, & R_3^-, & R_2, & R_1^+ & \quad \! \text{(in this
      order) to show that} \\
    u_{-1}, & f^*_2, & u_{-2}, & f^*_3 \upto  & u_{4-n}, & f^*_{n-3}, &
    u_{3-n}, & f^*_{n-2}, & u_{2-n}, & f^*_{n-1}, & f^*_n & \quad
    \text{(also in this order)}
  \end{array}
  \]
  lie in $\gfq\left[(f_1 \cdots f_n)^{-1},f_1 \upto f_n,f^*_1,u_0
    \upto u_{n-2}\right]$. So this algebra is equal to $B\left[(f_1
    \cdots f_n)^{-1}\right]$. We have shown that the
  hypothesis~\eqref{1lAd} in \lref{1lA} is satisfied, so
  \tref{main}\eqref{maina} follows. \\
  
  $\bullet$ Let us now prove~\eqref{mainb}. From \eqref{classical
    invariants} and \cref{1cField}, we get that $\gfq(V \oplus
  V^*)^{B_n}=\gfq(\fti_1,\dots,\fti_n,\fti_1^*,\dots,\fti_n^*,u_0)$,
  so the hypotheses~\eqref{1lAa} and~\eqref{1lAb} of \lref{1lA}
  follow.
  
  Now we analyze the relations~\eqref{3eqRels} in the same manner as
  in the proof of~\eqref{maina}. This results in
  the following table:
  \begin{center}
    \begin{tabular}{|c|c|c|c|c|c|}
      \hline \strut relation & involves & relation for &
      $\tilde{f}$-eliminates \strut[4mm] & $\tilde{f}^*$-eliminates &
      range \\ \hline
      \ref{3eqRk} & \parbox{2.5cm}{$\tilde{f}_1 \upto \tilde{f}_k$,
        \strut[4mm] $\tilde{f}^*_1 \upto \tilde{f}^*_{n+1-k}$, $u_{k-n}
        \upto u_{k-1}$} & $u_{2 k-n-1}$ & $\tilde{f}^*_{n+1-k}$ &
      $\tilde{f}_k$ & $1 \le k \le n$ \\ \hline
      \ref{3eqRkp} & \parbox{2.5cm}{$\tilde{f}_1 \upto \tilde{f}_k$,
        \strut[4mm] $\tilde{f}^*_1 \upto
        \tilde{f}^*_{n-k}$, $u_{k-n} \upto u_k$} & $u_{2 k-n}$ &
      $u_{k-n}$ & $u_k$ & $1 \le k \le n-1 $ \\
      \hline
    \end{tabular}
  \end{center}
  We first observe that the relations~\eqref{3eqRels} involve only the
  claimed generators. Secondly, there is one relation for each $u_i$,
  so the hypothesis~\eqref{1lAc} of \lref{1lA} is satisfied. Finally,
  to see that~\eqref{1lAd} is also satisfied, we use the relations
  \[
  \setlength{\arraycolsep}{0.25mm}
  \begin{array}{lllllllllllr}
    \tilde{R}_1, & \tilde{R}_1^+, & \tilde{R}_2, & \tilde{R}_2^+, &
    \tilde{R}_3, & \tilde{R}_3^+ \upto &  \tilde{R}_{n-2}, &
    \tilde{R}_{n-2}^+, & \tilde{R}_{n-1}, & \tilde{R}_{n-1}^+ &
    \tilde{R}_n & \quad \text{(in this order) to show that} \\
    \tilde{f}_1, & u_1, & \tilde{f}_2, & u_2, & \tilde{f}_3, & u_3
    \upto & \tilde{f}_{n-2}, & u_{n-2}, & \tilde{f}_{n-1}, & u_{n-1},
    & \tilde{f}_n & \quad \text{(also in this order)}
  \end{array}
  \]
  lie in $\gfq\left[(\tilde{f}^*_1 \cdots \tilde{f}^*_n)^{-1},\tilde{f}^*_1
    \upto \tilde{f}^*_n,u_{1-n} \upto u_0\right]$. We can also use
  \[
  \setlength{\arraycolsep}{0.25mm}
  \begin{array}{lllllllllllr}
    \tilde{R}_n, & \tilde{R}_{n-1}^+, & \tilde{R}_{n-1}, &
    \tilde{R}_{n-2}^+, & \tilde{R}_{n-2} \upto & \tilde{R}_3^+, &
    \tilde{R}_3, & \tilde{R}_2^+, & \tilde{R}_2, & \tilde{R}_1^+, &
    \tilde{R}_1 & \quad \text{(in this order) to show that} \\
    \tilde{f}^*_1, & u_{-1}, & \tilde{f}^*_2, & u_{-2}, & \tilde{f}^*_3 \upto &
    u_{3-n}, & \tilde{f}^*_{n-2}, & u_{2-n}, & \tilde{f}^*_{n-1}, &
    u_{1-n}, & \tilde{f}^*_n & \quad \text{(also in this order)}
  \end{array}
  \]
  lie in $\gfq\left[(\tilde{f}_1 \cdots \tilde{f}_n)^{-1},\tilde{f}_1
    \upto \tilde{f}_n,u_0 \upto u_{n-1}\right]$. This shows
  that~\eqref{1lAd} of \lref{1lA} is also satisfied, so applying the
  lemma yields the desired result.
  
  The statement about the minimality of generators will be proved
  below.
\end{proof}

\paragraph{Bigrading.} 
There is an obvious bigrading on $K[V \oplus V^*]$, given by assigning
the bidegree $(1,0)$ to every $x_i$ and $(0,1)$ to every $y_i$. This
bigrading passes to $K[V \oplus V^*]^G$ for every $G \le \GL_n(K)$. It
is interesting in itself, and also provides an easy way to prove the
minimality statement in \tref{main}.  All generating invariants
occurring in \tref{main} are bihomogeneous, and their bidegrees are
listed in the following table:

\begin{center}
  \begin{tabular}{c|c|c|c|c|c|}
    invariant & $f_i$ & $f_i^*$ & $\tilde{f}_i$ &
    $\tilde{f}^*_i$ & $u_i$ \\ \hline
    bidegree & $\left(q^{i-1},0\right)$ & $\left(0,q^{i-1}\right)$ &
    $\left((q-1) q^{i-1},0\right)$ & $\left(0,(q-1) q^{i-1}\right)$ & 
    $\begin{array}{ll}
      \left(q^i,1\right) & \strut[4mm] \text{if} \ i \ge 0, \\
      \left(1,q^{-i}\right) & \strut[4mm] \text{if} \ i \le 0
    \end{array}$
  \end{tabular}
\end{center}
The relations are also bihomogeneous of the following bidegrees:
\label{2RelDegs}

\begin{center}
  \begin{tabular}{c|c|c|c|c|c|}
    relation & $R_k$ & $R_k^+$ & $R_k^-$ &
    $\tilde{R}_k$ & $\tilde{R}_k^+$ \\ \hline
    \strut[4mm] bidegree \strut[4mm] & $\left(q^{k-1},q^{n-k}\right)$
    & $\left(q^k,q^{n-k}\right)$ & $\left(q^{k-1},q^{n+1-k}\right)$ &
    $(q-1) \cdot \left(q^{k-1},q^{n-k}\right)$ &
    $\left(q^k,q^{n-k}\right)$
  \end{tabular}
\end{center}

\begin{proof}[Proof of \tref{main} (continued)]
  To prove the minimality of the generating invariants, we assume, by
  way of contradiction, that one of the given generators is
  unnecessary. Then there exists a relation equating this generator to
  a polynomial in the other generators. We may assume this relation to
  be bihomogeneous of the same bidegree as the unnecessary generator.
  This implies that one of the generating relations must have bidegree
  bounded above (in both components) by the bidegree of the
  unnecessary generator. By comparing the bidegrees of the generating
  invariants and the bidegrees of the relations (and keeping in mind
  for which ranges of~$k$ each relation appears in \tref{main}), we
  see that this only happens in one case: if $q = 2$, then
  $\tilde{R}_1$ and $\tilde{R}_n$ have bidegrees
  $\left(1,q^{n-1}\right)$ and $\left(q^{n-1},1\right)$, respectively.
  Since this case was excluded in the minimality statement, the proof
  is complete.
\end{proof}

Since $\FF_q[V \oplus V^*]^{U_n}$ and $\FF_q[V \oplus V^*]^{B_n}$ are
complete intersections, we can also write down their bigraded Hilbert
series. For a general bigraded vector space $V$ (with
finite-dimensional bihomogeneous components $V_{d,e}$), the bigraded
Hilbert series is defined as
\[
H\left(V,s,t\right) := \sum_{d,e = 0}^\infty
\dim_K\left(V_{d,e}\right) s^d t^e \in \ZZ[[s,t]].
\]
The results are
\[
H\left(\FF_q[V \oplus V^*]^{U_n},s,t\right) =
\frac{\prod_{k=2}^{n-1}\left(1 - s^{q^{k-1}} t^{q^{n-k}}\right)
  \prod_{k=1}^{n-1} \left(1 - s^{q^k}
    t^{q^{n-k}}\right)}{\prod_{i=0}^{n-1}
  \left(\strut\left(1-s^{q^i}\right) \left(1 - t^{q^i}\right)\right)
  \prod_{i=0}^{n-2} \left(1 - s^{q^i} t\right) \prod_{i=1}^{n-2}
  \left(1 - s t^{q^i}\right)}
\]
and
\[
H\left(\FF_q[V \oplus V^*]^{B_n},s,t\right) = \frac{\prod_{k=1}^n
  \left(1 - s^{(q-1) q^{k-1}} t^{(q-1) q^{n-k}}\right)
  \prod_{k=1}^{n-1} \left(1 - s^{q^k}
    t^{q^{n-k}}\right)}{\prod_{i=0}^{n-1} \left(\strut\left(1 -
      s^{(q-1) q^i}\right) \left(1 - t^{(q-1) q^i}\right)\right)
  \prod_{i=0}^{n-1} \left(1 - s^{q^i} t\right) \prod_{i=1}^{n-1}
  \left(1 - s t^{q^i}\right)}.
\]
Notice that the Hilbert series with respect to the usual total degree
can be obtained from the bigraded Hilbert series by setting $s = t$.

\section{A conjecture about $\GL_n(\gfq)$}

We have also considered the invariant ring $\gfq[V \oplus
V^*]^{\GL_n(\gfq)}$ of the general linear group. It is well known that
the invariant ring $\gfq[V]^{\GL_n(\gfq)}$ is generated by the Dickson
invariants $c_{n,0} \upto c_{n,n-1}$ (see \mycite[Theorem~1.2]{wil} or
\mycite[Theorem~8.1.5]{lsm}). The $c_{n,i}^*$ are further invariants
in $\gfq[V \oplus V^*]^{\GL_n(\gfq)}$, and we also have the invariants
$u_i$. Various computations in the computer algebra system MAGMA
(see~[\citenumber{magma}]) have prompted us to make the following
conjecture.

\begin{conjecture} \label{3cGL}
  If $n \ge 2$, the invariant ring of the general linear group is
  generated by $4 n - 1$ invariants as follows:
  \[
  \gfq[V \oplus V^*]^{\GL_n(\gfq)} = \gfq[c_{n,0} \upto
  c_{n,n-1},c^*_{n,0} \upto c^*_{n,n-1},u_{1-n} \upto u_{n-1}].
  \]
  The invariant ring is Gorenstein but {\bfseries\itshape not} a
  complete intersection.
\end{conjecture}

We have been able to verify the conjecture computationally for $(n,q)
\in \{(2,2),(2,3),(2,4),\linebreak (3,2)\}$.  For $(n,q) \in
\left\{(2,5),(2,7),(3,3),(4,2)\right\}$, we managed to gain evidence
for the conjecture by checking that all invariants up to some degree
(as far as the computer calculation was possible) lie in the algebra
that Conjecture~\ref{3cGL} claims to be the invariant ring.

\tref{main} and Conjecture~\ref{3cGL} (if true)
tell us that for $G \in \left\{U_n,B_n,\GL_n(\gfq)\right\}$,
the invariant ring $\gfq[V \oplus V^*]^G$ is generated by generators
of $\gfq[V]^G$, their $*$-images, and invariants of the form
$u_i$. How general is this phenomenon? To find out, we considered the
special linear groups.

\begin{ex} \label{3exSL}
  For $G = \SL_2(\FF_{\! 3})$ and $V = \FF_3^2$ the natural
  $G$-module, we have
  \[
  \FF_{\! 3}[V]^G = \FF_{\! 3}\left[\strut\right.\underbrace{x_1^3 x_2
    - x_1 x_2^3}_{=: f_1},\underbrace{x_1^6 + x_1^4 x_2^2 + x_1^2
    x_2^4 + x_2^6}_{=: f_2}\left.\strut\right].
  \]
  (In fact, the invariants of $\SL_n(\gfq)$ acting on its natural
  module are well known for general~$n$ and~$q$, see
  \mycite[Theorem~8.1.8]{lsm}.) Turning to the action on $\FF_{\! 3}[V
  \oplus V^*]$, we verify that the $G$-orbit of $h := x_1 y_2 - x_2
  y_1$ has length~6 and includes $-h$. Therefore a square root of the
  negative of the orbit product is an invariant, which we write as~$g
  \in \FF_{\! 3}[V \oplus V^*]^G$. The bidegree of~$g$ is $(3,3)$.
  
  On the other hand, the $f_i$ and their $*$-images have bidegrees
  $(4,0),(6,0),(0,4)$, and $(0,6)$, and the $u_i$ and $u_{-i}$ have
  bidegrees $(3^i,1)$ and $(1,3^i)$, respectively, for~$i$
  non-negative. So $g \in
  \FF_{\! 3}[f_1,f_2,f_1^*,f_2^*,u_0,u_1,u_{-1},\ldots]$ would
  imply $g = \pm u_0^3$, which is not the case. We conclude that for
  $G = \SL_2(\FF_{\! 3})$, the invariant ring $\FF_{\! 3}[V \oplus V^*]^G$ is
  {\em not} generated by generators of $\FF_{\! 3}[V]^G$, their
  $*$-images, and invariants of the form $u_i$.
  
  Further calculations show that this carries over to other special
  linear groups $\SL_n(\gfq)$.
\end{ex}

\bigskip

\begin{center}
\begin{tabular}{ll}
  C\'edric Bonnaf\'e & Gregor Kemper \\
  Universit\'e Montpellier 2 & Technische Universit\"at
  M\"unchen \\
  Institut de Math\'ematiques et de Mod\'elisation  & Zentrum Mathematik - M11 \\
  de Montpellier (CNRS-UMR 5149) & \\
  Place Eug\`ene Bataillon & Boltzmannstr. 3 \\
  34095 MONTPELLIER Cedex & 85\,748 Garching \\
  France & Germany \\
  {\tt cedric.bonnafe$@$math.univ-montp2.fr} & {\tt
    kemper$@$ma.tum.de}
\end{tabular}
\end{center}

\end{document}